\documentclass{amsart}

\usepackage[all]{xy}
\usepackage{amssymb}
\usepackage{amsmath}
\usepackage{amsfonts}
\usepackage{epsfig}
\usepackage{tikz}
\newtheorem{theorem}{Theorem}[section]
\newtheorem{prop}[theorem]{Proposition}
\newtheorem{lemma}[theorem]{Lemma}
\newtheorem{cor}[theorem]{Corollary}
\theoremstyle{definition}
\newtheorem{defi}[theorem]{Definition}
\newtheorem{rem}[theorem]{Remark}

\newcommand{\ihsk}{$\mathrm{IHS}-K3^{[2]}$ } 
\newcommand{\ihskpt}{$\mathrm{IHS}-K3^{[2]}$. } 
\newcommand{\ihsknpt}{$\mathrm{IHS}-K3^{[n]}$. } 
\newcommand{\ihskcom}{$\mathrm{IHS}-K3^{[2]}$, } 
\newcommand{\ihskncom}{$\mathrm{IHS}-K3^{[n]}$, } 
\newcommand{\ihskn}{$\mathrm{IHS}-K3^{[n]}$ } 
\newcommand{\ihsknpts}{$\mathrm{IHS}-K3^{[n]}$: } 
\newcommand{\lra}{\longrightarrow}

\newcommand{\ie}{i.e. }

\newcommand{\ra}{\rightarrow}
\newcommand{\IP}{\mathbb{P}}
\newcommand{\IC}{\mathbb{C}}
\newcommand{\IR}{\mathbb{R}}

\newcommand{\IZ}{\mathbb{Z}}
\newcommand{\IQ}{\mathbb{Q}}
\newcommand{\IN}{\mathbb{N}}

\newcommand{\cC}{\mathcal{C}}
\newcommand{\cK}{\mathcal{K}}
\newcommand{\cP}{\mathcal{P}}
\newcommand{\cM}{\mathcal{M}}

\newcommand{\coloneqq}{:=}

\DeclareMathOperator{\rank}{\rm{rank}}
\DeclareMathOperator{\rk}{\rm{rank}}

\DeclareMathOperator{\Aut}{\rm Aut}
\DeclareMathOperator{\Bir}{\rm Bir}

\DeclareMathOperator{\sign}{\rm sign}
\DeclareMathOperator{\Hom}{\rm Hom}

\DeclareMathOperator{\id}{\rm id}
\DeclareMathOperator{\NS}{\rm NS}

\DeclareMathOperator{\Trans}{\rm Transc}

\DeclareMathOperator{\GL}{\rm GL}
\DeclareMathOperator{\Mo}{\rm Mon}
\DeclareMathOperator{\discr}{\rm discr}

\begin{document}

\date{\today}
\title[Complex ball quotients]{Complex ball quotients from manifolds of $K3^{[n]}$-type}
\author{Samuel Boissi\`ere, Chiara Camere and Alessandra Sarti}

\address{Samuel Boissi\`ere, Universit\'e de Poitiers, 
Laboratoire de Math\'ematiques et Applications, 
 T\'el\'eport 2 
Boulevard Marie et Pierre Curie
 BP 30179,
86962 Futuroscope Chasseneuil Cedex, France}
\email{samuel.boissiere@math.univ-poitiers.fr}
\urladdr{http://www-math.sp2mi.univ-poitiers.fr/~sboissie/}

\address{Chiara Camere, Universit\`a  degli Studi di Milano,
Dipartimento di Matematica,
Via Cesare Saldini 50,
20133 Milano, Italy} 
\email{chiara.camere@unimi.it}
\urladdr{http://www.mat.unimi.it/users/camere/en/index.html}

\address{Alessandra Sarti, Universit\'e de Poitiers, 
Laboratoire de Math\'ematiques et Applications, 
 T\'el\'eport 2 
Boulevard Marie et Pierre Curie
 BP 30179,
86962 Futuroscope Chasseneuil Cedex, France}
\email{sarti@math.univ-poitiers.fr}
\urladdr{http://www-math.sp2mi.univ-poitiers.fr/~sarti/}



\begin{abstract}
We describe periods of irreducible holomorphic manifolds of $K3^{[n]}$-type with a non-symplectic automorphism of prime order $p\geq 3$. These turn out to lie on complex ball quotients and we are able to give a precise characterization of when the period map is bijective, by introducing the notion of $K(T)$-generality.
\end{abstract}

\maketitle

\section{Introduction}

We classified in \cite{BCS} all non-symplectic automorphisms of prime order $p$ acting on \ihskcom i.e. fourfolds which are deformation equivalent to the Hilbert scheme of two points of a smooth $K3$ surface, for $p=2,3$ and $7\leq p\leq 19$. Our classification relates certain invariants of the fixed locus with the isometry classes of two natural lattices, associated to the action of the automorphism on the second integral cohomology group. Then, in \cite{BCMS} and in \cite{kevin} the cases $p=23$ and $p=5$ were solved, thus completing the classification for \ihsk and prime order.

Here, we want to parametrize IHS manifolds which admit an action of a given non-symplectic automorphism of prime order $p$. For this we use its action on the second cohomology: given $\sigma$ acting on $X$, $\sigma^*$ acts on $H^2(X,\IZ)$ as a monodromy operator which is a Hodge isometry and preserves a K\"ahler class.

If the only automorphism acting trivially on cohomology is the identity (satisfied for $K3^{[n]}$-type by \cite[Proposition 10]{BeauvilleKaehler}), then the monodromy $\sigma^*$ reconstructs $\sigma$ by the Hodge-theoretic Torelli Theorem \ref{HTT}. In particular, given a Hodge monodromy that preserves a K\"ahler class, it lifts to exactly one automorphism.

In the case of \ihskcom we know from \cite{BCS} that the action on cohomology is classified once given three numerical invariants $(p,m,a)$, or equivalently once given the invariant sublattice $T$ inside the second cohomology lattice $L$ (compare with \cite[Corollary 5.7]{BCS}). For the other deformation classes this is not known yet, even for \ihsknpts the invariant sublattice $T$ is a necessary information, but may not be sufficient to determine completely the action. That is why we need to fix the isometry in $O(L)$ representing the automorphism, and this leads us to the study of $(\rho,T)$-polarizations, as defined in \S \ref{definitions}.

The study of the moduli spaces of projective irreducible holomorphic symplectic manifolds (IHS for
short) was started by Gritsenko, Hulek and Sankaran in \cite{GKS}, where the authors consider polarized  IHS manifolds and they show that, for IHS manifolds that are deformations of the Hilbert scheme of $n$ points on a $K3$ surface (we say that these are $\mathrm{IHS}-K3^{[n]}$),  if the polarization has degree large enough, then the corresponding moduli space is algebraic, and in fact of general type. 

On the other hand, when the Picard rank of the considered projective family grows, the period map is a priori non-injective, because of the existence of non-isomorphic birational models in its fibres.
In a recent work, Amerik and Verbitsky \cite{AmerikVerbitsky} were able to give a precise description of the K\"ahler cone of an IHS manifold. Their results are fundamental to start the description of the moduli space of IHS manifolds with a non-symplectic automorphism and were first applied by Joumaah in \cite{Joumaah} to describe the moduli space of \ihskn with a non-symplectic involution. 

In this paper, we generalize to \ihskn the construction of Dolgachev and Kond\=o 
\cite{DK} of the moduli space of $K3$ surfaces with a non-symplectic automorphism of prime order $p\geq 3$. By using results of \cite{BCS}, we first construct in \S \ref{surjective-period} a surjective period map to the complement of a hyperplane arrangement inside a complex ball; these hyperplanes are the analogous for \ihskn manifolds of the hyperplanes determined by $(-2)$-curves in similar moduli problems for $K3$ surfaces, see \cite{DK}. Then, in \S \ref{injectivityperiod}, by using the notion of $K(T)$-generality,  we are able to exhibit a bijective period map. Finally, in \S \ref{arithmetic-quot} we obtain a quasi-projective variety parametrizing isomorphism classes of $K(T)$-general \ihsknpt

\subsection*{Acknowledgements}

The second named author was partially supported by the Research Network Program GDRE-GRIFGA. The authors would like to thank Prof. Bert van Geemen for an enlightening conversation and Prof. Shigeyuki Kond\=o for his suggestion about moduli spaces of cubic threefolds. 

\section{Preliminary notions}
\subsection{Lattices}

A {\it lattice} $L$ is a free $\IZ$-module equipped with a non-degenerate symmetric bilinear form
$\langle \cdot, \cdot\rangle$ with integer values.  Its {\it dual lattice} is 
$L^{\vee}\coloneqq\Hom_{\IZ}(L,\IZ)$. It can be also described as follows:
$$
L^{\vee}\cong\{x\in L\otimes \IQ~|~\langle x,v\rangle\in \IZ\quad \forall v\in L\}.
$$
Clearly $L$ is a sublattice of $ L^{\vee}$ of the same rank, so the \emph{discriminant group} ${A_L:=L^{\vee}/L}$ is a finite abelian group whose order is denoted $\discr(L)$ and called the {\it discriminant of $L$}. We denote by
$\ell(A_L)$ the \emph{length} of $A_L$, \ie the minimal number of generators of $A_L$.  Let $\{e_i\}_i$ be a basis of~$L$ and 
$M\coloneqq(\langle e_i,e_j\rangle)_{i,j}$ the Gram matrix, then one has $\discr(L)=|\det(M)|$.

A lattice $L$ is called \emph{even} if $\langle x,x\rangle\in 2\IZ$ for all $x\in L$.  In this case   
the bilinear form induces a  quadratic form $q_L: A_L\lra \IQ/2\IZ$. Denoting by $(s_{(+)},s_{(-)})$ the signature of
$L\otimes\IR$, the triple of invariants $(s_{(+)},s_{(-)},q_L)$ characterizes the \emph{genus} of the even lattice $L$ (see \cite[Chapter 15, \S 7]{conwaysloane}, 
\cite[Corollary 1.9.4]{Nikulinintegral}).

 A sublattice $M\subset L$ is called \emph{primitive} if $L/M$ is a free $\IZ$-module. Let $p$ be a prime number. A lattice $L$ is called $p$-\emph{elementary} if $A_L\cong\left(\frac{\IZ}{p\IZ}\right)^{\oplus a}$ for some
non negative integer $a$ (also called the \emph{length} $\ell(A_L)$ of $A$).  We write $\frac{\IZ}{p\IZ}(\alpha)$, $\alpha\in\IQ/2\IZ$ to denote that the quadratic form $q_L$ takes value $\alpha$ on the generator 
of the $\frac{\IZ}{p\IZ}$ component of the discriminant group.

We  denote by $U$ the unique even unimodular hyperbolic 
lattice of rank two and by $A_k, D_h, E_l$ the even, negative definite lattices associated to the Dynkin diagrams 
of the corresponding type ($k\geq 1$, $h\geq 4$, $l=6,7,8$). We denote by $L(t)$ the lattice
whose bilinear form is the one on $L$ multiplied by $t\in\IN^\ast$. 

In the sequel we will be using the lattice $E_6^\vee(3)$ (see \cite{AST}):
it is even, negative definite and $3$-elementary with $a=5$. To get a simple form of its discriminant group one can proceed as follows. By~\cite[Table 2]{AS} the lattice $U(3)\oplus E_6^\vee(3)$ admits a primitive
embedding in the unimodular $K3$ lattice with orthogonal complement isometric to $U\oplus U(3)\oplus A_2^{\oplus 5}$. It follows that the discriminant form of $E_6^\vee(3)$ is the opposite of those of $A_2^{\oplus 5}$, so it is $\IZ/3\IZ(2/3)^{\oplus 5}$.

\subsection{IHS manifolds and their moduli spaces}

A compact complex K\"{a}hler manifold $ X $ is {\it irreducible holomorphic symplectic} (IHS) if it is simply connected and admits a holomorphic  $2$-form $ \omega_X \in H^{2,0}(X) $ everywhere non degenerate and unique up to multiplication by a non-zero scalar.
The existence of such a symplectic form $\omega_X$ immediately implies that the dimension of $X $ is an even integer. Moreover, the canonical divisor $ K_X$ is trivial, $c_1(X)=0$, and 
 $ T_X \cong \Omega _X^1 $. For a complete survey of this topic we refer the reader to the book \cite{GrossJoyceHuy} and references therein.  The second cohomology group $H^2(X,\IZ)$ is an integral lattice for the Beauville--Bogomolov--Fujiki quadratic form see \cite{Beauvillec1Nul}.

One of the most studied deformation families is that of $ X=S^{\left[n\right]} $, with $n\geq 2$, the Hilbert scheme of $0$-dimensional subschemes of length $n$ of  a smooth $K3$ surface $ S $.
The lattice $(H^2(X,\mathbb{Z}),q)$ in this case is $L=U^{\oplus 3}\oplus E_8^{\oplus 2}\oplus \langle -2(n-1)\rangle$;
we say that an IHS manifold $X$ is an \ihskn if it is deformation equivalent to the Hilbert scheme of $n$ points on a $K3$ surface. 

We recall some well known facts from \cite{Huybrechts} and \cite{MarkmanTorelli}. If $X$ is an IHS manifold, a {\it marking} for $X$ is 
an isometry $\eta: L \lra H^2(X,\IZ)$; the manifold $X$ is sometimes said to be {\it of type $L$}. An isomorphism $f:X_1\lra X_2$ is an isomorphism of marked pairs $(X_1,\eta_1)$ and $(X_2,\eta_2)$ if $\eta_1=f^*\circ \eta_2$. There exists a coarse moduli space $\mathcal{M}_{L}$ that parametrizes isomorphism classes of marked pairs of type $L$; it is a non-Hausdorff smooth complex manifold (see \cite{Huybrechts}). If $X$ is an \ihskn then  $\mathcal{M}_{L}$ has dimension $21$. Denote by 
$$
\Omega_L:=\{\omega\in\IP(L\otimes \IC)\, |\, q(\omega)=0,\, q(\omega+\bar{\omega})>0\}
$$
the {\it period domain}; it is an open (in the analytic topology) subset of the non-singular quadric defined by $q(\omega)=0$. The period map
$$
\mathcal{P}:\mathcal{M}_{L} \lra \Omega_L, (X,\eta)\mapsto \eta^{-1}(H^{2,0}(X))
$$ 
is a local isomorphism by the Local Torelli Theorem \cite[Th\'eor\`eme 5]{BeauvilleKaehler}. For $\omega\in \Omega_L$ we consider
$$
L^{1,1}(\omega):=\{\lambda\in L\,|\, (\lambda,\omega)=0\},
$$
where $(\cdot,\cdot)$ is the bilinear form associated to the quadratic form $q$. 
Then $L^{1,1}(\omega)$ is a sublattice of $L$,
and, given a marked pair $(X,\eta)$, we get  $\eta^{-1}(\NS(X))= L^{1,1}(P(X,\eta))$. The set
$\{\alpha\in H^{1,1}(X)\cap H^2(X,\IR)\,|\, q(\alpha) >0\}$  has two connected components; the {\it positive cone} $\mathcal{C}_X$ is the connected component containing the {\it K\"ahler cone $\mathcal{K}_X$}.

Recall that two points $x,y$ of a topological space $M$ are called {\it inseparable} if every pair of open neighbourhoods $x\in U$ and $y\in V$ has non-empty intersection; a point $x\in M$ is called a {\it Hausdorff point} if for every $y\in M$, $y\not= x$, then $x$ and $y$ are separable. 

\begin{theorem}[Global Torelli Theorem] \cite{Verbitsky},\cite[Theorem 2.2]{MarkmanTorelli} \label{GTT}

Let $\mathcal{M}^0_L$  be a connected component of $\mathcal{M}_L$.
\begin{enumerate}
\item The period map $\mathcal{P}$ restricts to a surjective holomorphic map $$\mathcal{P}:\mathcal{M}^0_L\lra \Omega_L.$$ (we call it again $\mathcal{P}$ for simplicity).
\item For each $\omega\in\Omega_L$, the fiber $\mathcal{P}^{-1}(\omega)$ consists of pairwise inseparable points.   
\item Let $(X_1,\eta_1)$ and $(X_2,\eta_2)$ be two inseparable points of $\mathcal{M}_L^0$. Then $X_1$ and~$X_2$ are bimeromorphic. 
\item The point $(X,\eta)\in\mathcal{M}_L^0$ is Hausdorff  if and only if $\mathcal{C}_X=\mathcal{K}_X$.
\end{enumerate}
\end{theorem}

In the sequel we will be using also the following Hodge theoretic version of the Torelli theorem.

\begin{theorem}\cite[Theorem 1.3]{MarkmanTorelli}\label{HTT}. 
 Let $X$ and $Y$ be two irreducible holomorphic symplectic manifolds deformation equivalent one to each other. Then:
 \begin{enumerate}
  \item $X$ and $Y$ are bimeromorphic if and only if there exists a parallel transport operator $f:H^2(X,\IZ)\ra H^2(Y,\IZ)$ that is an isomorphism of integral Hodge structures;
  \item if this is the case, there exists an isomorphim $\tilde{f}:X\ra Y$ inducing $f$ if and only if $f$ preserves a K\"ahler class.
 \end{enumerate}
\end{theorem} 
Recall that a  parallel transport operator $f:H^*(X,\IZ)\lra H^*(Y,\IZ)$ is called a {\it monodromy operator}. If $\Mo (X)\subset \GL(H^*(X,\IZ))$ denotes the subgroup of monodromy operators, we denote by $\Mo^2(X)$ its image in $O(H^2(X,\IZ))$. So we have

\begin{defi}
 Given a marked pair $(X,\eta)$ of type $L$, we define the \textit{monodromy group} as $\Mo^2(L):=\{\eta^{-1}\circ f\circ\eta\,|\, f\in \Mo^2(X)\}\subset \GL(L)$. 
\end{defi}

A priori, the definition of $\Mo^2(L)$ depends on the choice of $(X,\eta)$, but it is easy to show that it is well-defined on the connected component $\mathcal{M}^0_L$. It was proven by Verbitsky in \cite{VerTorelli} that $\Mo^2(L)$ is an arithmetic subgroup of $O(L)$.

By a result of Markman \cite[Theorem 1.2]{Markmanintegral}, if $X$ is \ihskncom then $\Mo^2(X)$ is a normal subgroup of $O(H^2(X,\IZ))$. In particular, if $n=2$ then
$\Mo^2(X)=O^{+}(H^2(X,\IZ))$, which are the isometries of $H^2(X,\IZ)$ that preserve the positive cone, so in this case Theorem \ref{HTT} can be restated in the following way (which is essentially the same statement as for $K3$ surfaces):\\
 Let $X$ be an \ihskpt Then:
 \begin{enumerate}
  \item Let  $h\in O^+(H^2(X,\IZ))$ be an isomorphism of integral Hodge structures, then there exists 
$f\in \Bir(X)$, the group of birational transformations of $X$, such that $f^*=h$;
  \item Let $h\in O^+(H^2(X,\IZ))$ be an isomorphism of integral Hodge structures. There exists $f\in\Aut(X)$ such that $f^*=h$  if and only if $h$ preserves a K\"ahler class.
 \end{enumerate}


\section{Non-symplectic automorphisms of IHS manifolds and $(\rho, T)$-polarizations}

We briefly review here what is known for non-symplectic automorphisms of IHS manifolds. Let $X$ be an IHS manifold and $f$ be a holomorphic automorphism of $X$ of prime order $p$ acting non-symplectically: $f^*$ acts on $H^{2,0}(X)$ by multiplication by a primitive $p$-th root of the unity. Such automorphisms can exist only when $X$ is projective. 
It follows that the invariant lattice $T\subset H^2(X,\IZ)$  is a primitive sublattice of the N\'eron--Severi group $\NS(X)$,
and consequently the characteristic polynomial of the action of $f$ on the transcendental lattice $\Trans(X)$ is the $k$-th power of the $p$-th cyclotomic polynomial $\Phi_p$. Thus $k\varphi(p)=k(p-1)=\rank_\IZ\Trans(X)$, and in particular 
$$
\varphi(p)\leq b_2(X)-\rho(X),
$$
where $\varphi$ is the Euler's totient function and $\rho(X)=\rank_\IZ\NS(X)$ is the Picard number of $X$. 
If $X$ is \ihskncom since $b_2(X)=23$, the maximal prime order for $f$ is $p=23$, and this can happen only when $\rho(X)=1$.

Observe that a very general projective \ihsk has no non trivial automorphisms. Here, by {\it very general} we mean that the manifold has Picard number one and it is not a special member in the moduli space. We believe that this fact is well known, but since we could not find an explicit proof in the literature, we give here a proof that uses our previous results:
\begin{theorem}
Let $X$ be a very general projective \ihsk, $\NS(X)\not=\langle 2\rangle$, then $\Aut(X)=\id$. 
\end{theorem}
\begin{proof}
A generic projective IHS manifold $X$ has N\'eron--Severi group of rank $1$ equal to $\langle 2t \rangle$, and so $\rk \Trans(X)=22$. Since an automorphism of $X$ induces a Hodge isometry on $H^2(X,\IZ)$, it preserves $\NS(X)$ and so it preserves an ample class. This means that every element of $\Aut(X)$ preserves a K\"ahler metric, hence it is an isometry.  In conclusion we have that $\Aut(X)$ is a discrete Lie subgroup of a compact group, so it is finite. We are now left to study automorphisms of finite order on $X$, and it is easy to show that we can restrict to the case of prime order.

Recall that if $X$ admits a symplectic automorphism then $\rk \NS(X)\geq 8$, see \cite[Section 6.2]{mongardiPhD}, so we do not have such automorphisms. If $\sigma$ is non-symplectic of prime order $p$, then $p-1$ must divide $22$. So we have the possibilities $p=2, 3, 23$. For $p=23$ then $\NS(X)=\langle 46 \rangle$, and only a very special \ihsk carries an order $23$ non-symplectic automorphism as shown in \cite{BCMS}. For $p=3$ the only possibility is $\NS(X)=\langle 6 \rangle$, and for these \ihsk we do not always have a non-symplectic automorphism, see \cite{BCS}. If $p=2$, then $\NS(X)= \langle 2 \rangle$ by \cite{BCS}, and this case corresponds to an \ihsk with an involution that deforms to Beauville's involution on the Hilbert scheme of two points of a quartic in $\IP^3$ not containing a line. 
\end{proof}

\subsection{$(\rho,T)$-polarized marked pairs}\label{definitions}

Let now $T$ be an even non-degenerate lattice of rank  $r\geq 1$ and signature $(1,r-1)$.
A \emph{$T$-polarized} IHS manifold is a pair $(X,\iota)$,  where $X$ is a projective IHS manifold and $\iota$ is a primitive embedding of lattices $\iota:T\hookrightarrow \NS(X)$ (see also \cite{C4}). Observe that we are then assuming that $T$ has a primitive embedding in $L$,
and we identify $T$ with its image as sublattice of $L$.

Let $(\bar{X},\bar{\iota})$ be a $T$-polarized IHS manifold such that there exists a cyclic group $G=\langle \bar{\sigma}\rangle\subset \Aut(\bar{X})$ of prime order $p\geq 3$ acting non-symplectically on $\bar{X}$. Assume that the action of $G$ on $\bar{\iota}(T)$ is the identity and that there exists a group homomorphism $\rho: G\longrightarrow O(L)$ such that
$$
T=L^{\rho}:=\{x\in L\,|\, \rho(\bar \sigma)(x)=x\}.
$$

\begin{defi}
A {\it $(\rho, T)$-polarization} of a $T$-polarized $(X,\iota)$ is a marking $\eta\colon L\to H^2(X,\IZ)$ such that $\eta_{|T}=\iota$ and such that there exists $\sigma\in\Aut(X)$ satisfying $\sigma^\ast=\eta\circ\rho({\bar\sigma})\circ\eta^{-1}$.

The pair $(X,\eta)$ is said to be {\it $(\rho, T)$-polarized} (in order to keep a light notation, we forget about $\iota$, though it is part of the data). 
\end{defi}

\begin{rem}\label{rho-monodromy}
 It follows immediately from the definition and from the Hodge-theoretic Torelli Theorem \ref{HTT} that a necessary condition for the existence of $(\rho,T)$-polarized marked pairs is that $\rho(\bar{\sigma})\in \Mo^2(L)$. 
\end{rem}

Two $(\rho, T)$-polarized marked IHS manifolds $(X_1, \eta_1)$ and $(X_2,\eta_2)$ are isomorphic if there is an isomorphism $f\colon X_1\to X_2$ such that $\eta_1=f^*\circ \eta_2$.

Let $\omega$ be the line in $L\otimes \IC$ defined by $\omega=\eta^{-1}(H^{2,0}(X))$ and let  $\xi\in\IC^*$ be such that $\rho(\bar{\sigma})(\omega)=\xi \omega$. Observe that $\xi\not= 1$, since the action is non-symplectic, and $\xi$ is a primitive $p$-th root of unity not equal to $-1$ since $p$ is a prime number $p\geq 3$. The period $\omega$ belongs to the eigenspace $S(\xi)$ of $S\otimes \IC$ relative to the eigenvalue $\xi$, where $S$ is the orthogonal complement of 
$T$ in $L$.  
 
Then the period belongs to the space
$$
\Omega_T^{\rho,\xi}:=\{x\in \IP(S(\xi))\,|\, q(x+\bar{x})>0\}
$$ 
of dimension  $\dim S(\xi)-1$, which is a complex ball if $\dim S(\xi)\geq 2$. It is easy to check that every point $x\in \Omega_T^{\rho,\xi}$  satisfies automatically the condition $q(x)=0$. 

We recall now some of the results of \cite{BCS} which contains the classification of non--symplectic automorphisms of prime order $3\leq p\leq 19$, $p\not=5$ and partial results on involutions, completing the ones in \cite{BeauvilleInv}. The cases $p=5$ and $p=23$ were then discussed respectively in \cite{kevin} and in \cite{BCMS}. Such automorphisms are classified in terms of their invariant sublattice $T$ (see \cite[Appendix A]{BCS} and \cite[Section 3.4]{kevin}). For higher dimensional \ihskn a classification is not known yet, there are only partial results due to \cite{Joumaah} in the case of involutions.

Moduli spaces of \ihskn endowed with a non-symplectic involution are studied in \cite{Joumaah}; in the case of \ihsk, for odd primes, some partial results were contained in \cite{BCS}.
There the authors deal with the case $T=\overline{T}\oplus \langle -2\rangle$, where $\overline{T}$ is an even non degenerate lattice of signature $(1,21-(p-1) m)$ with fixed primitive embedding in the $K3$ lattice  $\Lambda$, and $S=T^{\perp}\cap L\subset \Lambda$. 
 
\begin{theorem}{\cite[Theorem 5.5]{BCS}}\label{ball}
Let $X$ be a $(\rho, T)$-polarized \ihsk such that $H^{2,0}(X)$ is contained in the eigenspace of $H^2(X,\IC)$ relative to $\xi$ (with $T$ with a decomposition as above). 
 Then $\omega_X\in \Omega_T^{\rho,\xi}$ and conversely, if  $\dim S(\xi)\geq 2$ every point of $\Omega_T^{\rho,\xi}\setminus\bigcup_{\delta\in S, q(\delta)=-2} \left(H_{\delta}\cap \IP(S(\xi))\right)$ is the period point of some $(\rho, T)$-polarized \ihsk (where $H_\delta$ is the hyperplane in $\IP(S(\xi))$ orthogonal to $\delta$).
\end{theorem}

A comparison with \cite[Appendix A]{BCS} shows that there are only two cases in the classification of non-symplectic automorphisms of prime order on \ihsk where the assumption on $T$ of Theorem \ref{ball} is not satisfied, namely $T=\langle 6\rangle$ and $T=\langle 6\rangle \oplus E_6^*(3)$.

\section{The image of the period map}\label{surjective-period}

The aim of this section is now to compute which are the period points, via the period map $\mathcal{P}$, corresponding to \ihskn which are
$(\rho,T)$-polarized for all primes $p\geq 3$, with given invariant lattice $T$; in particular, in the case of \ihskcom we consider $T$ as classified in \cite[Appendix A]{BCS}.

Recall the following definition (see \cite[Definition 1.13]{AmerikVerbitsky} and also \cite[Definition 5.10]{MarkmanTorelli}):
\begin{defi}
Let $X$ be an IHS manifold. A rational non-zero class $\delta\in H^{1,1}(X)\cap H^2(X,\IQ)$ with $q(\delta)<0$ is said to be \textit{monodromy birationally minimal} (MBM) if there exists a bimeromorphic map $f:X\dashrightarrow Y $ and a monodromy operator $g\in\Mo^2(X)$ which is also a Hodge isometry such that the hyperplane $\delta^{\perp}\subset H^{1,1}(X)\cap H^2(X,\IR)$ contains a face of $g(f^*(\mathcal{K}_{Y}))$. Let $\Delta(X)$ be the set of integral MBM classes $\delta\in H^{1,1}(X)\cap H^2(X,\IZ)$ on $X$.
\end{defi}

We call the classes in $\Delta(X)$ {\it wall divisors} (see also \cite{Knutsen-Lelli-Chiesa-Mongardi}). An essential result for what follows is:

\begin{theorem}\cite[Theorem 6.2]{AmerikVerbitsky}\label{wallsMBM}
 Let $X$ be an IHS manifold, $\Delta(X)$ as above and 
 $$
 \mathcal{H}:=\bigcup_{\delta\in \Delta(X)} \delta^{\perp}\subset H^{1,1}(X)\cap H^2(X,\IR)
 $$

 Then the K\"ahler cone of $\mathcal{K}_X$ is a connected component of $\mathcal{C}_X\setminus\mathcal{H}$.

\end{theorem}

The previous theorem generalizes the analogous result for $K3$ surfaces, where 
MBM classes replace the $(-2)$--curves. Recall also the following:

\begin{defi}\cite[Definition 6.1]{AmerikVerbitsky}
A {\it K\"ahler-Weyl chamber} of $X$ is the image $g(f^*\mathcal{K}_Y)$ of the K\"ahler cone of $Y$ under some $g\in \Mo^2(X)$ which is also an isomorphism of Hodge structures, where $Y$ runs through  all birational models of $X$ and $f:X\dashrightarrow Y$.

\end{defi}
\begin{rem}\label{utile}
\begin{itemize}
\item[1)] By \cite[Lemma 5.12]{MarkmanTorelli} if $X_1$ and $X_2$ are birational, then 
the birational map defines a parallel transport operator which is an Hodge isometry. This maps K\"ahler-Weyl chambers in $\cC_{X_1}$ to  K\"ahler-Weyl chambers in $\cC_{X_2}$. In particular the number of connected components in Theorem \ref{wallsMBM} does not depend on the birational model we have chosen. 
\item[2)] Fix an MBM class $\delta\in\NS(X)$, then as seen in the definition $\delta^{\perp}$
does not necessarily contain  a face of the K\"ahler cone $\cK_X$ of $X$ but it has constant sign on it, i.e. $(\delta,k)>0$ for all $k\in
\cK_X$ or $(-\delta,k)>0$ for all $k\in\cK_X$. In particular $\delta$ can not be zero on $\cK_X$, otherwise $\cK_X\subset \delta^\perp$, which is not possible by definition of MBM class.
\end{itemize}
\end{rem}

Let $\mathcal{M}_L^+$ be a connected component of the moduli space of marked \ihskn and from now on let $\mathcal{P}$ be the restriction of the period map to a connected component  $\mathcal{M}_L^+\ra \Omega_L$. 
Define $\Delta(L)$ as the set of $\bar{\delta}\in L$ such that there exists $(X,\phi)\in \mathcal{M}_L^+$ with $\phi(\bar{\delta})\in \NS(X)$ a MBM class. Observe that by \cite[Theorem 2.16]{AmerikVerbitskyMK} we have $\Delta(X)=\phi(\Delta(L))\cap \NS(X)$. We denote $\Delta(S)=\Delta(L)\cap S$.

\begin{theorem}\label{surjectivity}
Let $T\subset L$ be a fixed primitive embedding and $\rho:G\ra O(L)$ be  a group homomorphism such that there exists a $(\rho,T)$-polarized \ihskn $(\bar{X},\bar{\phi})\in\mathcal{M}_L^+$.
\begin{enumerate}
 \item \label{surjectivity-i} Let $(X,\phi)\in \mathcal{M}_L^+$ be a $(\rho, T)$-polarized \ihskn such that $H^{2,0}(X)$ is contained in the eigenspace of $H^2(X,\IC)$ relative to $\xi$. 
 Then $	\mathcal{P}(X,\phi)\in \Omega_T^{\rho,\xi}\setminus\Delta$,  where
\[
\Delta:=\bigcup_{\delta\in \Delta(S)} (H_{\delta}\cap \Omega_T^{\rho,\xi}).
\] 
and $H_{\delta}$ is the hyperplane orthogonal to $\delta$ in $\IP(L_{\IC})$. 
\item \label{surjectivity-ii} Conversely, every point of $\Omega_T^{\rho,\xi}\setminus\Delta$ is the period point of some $(\rho, T)$-polarized \ihskn.
\end{enumerate}

 \end{theorem}

\begin{proof}
Given a $(\rho, T)$-polarized $(X,\phi)\in\mathcal{M}_L^+$, where $\phi$ denotes the marking, let $\omega:=\phi^{-1}(H^{2,0}(X))$ be its period. If we had $\omega\in H_{\delta}$ for a $\delta \in \Delta(S)$, then, by definition of $\NS(X)$ as orthogonal complement of $H^{2,0}(X)$ in $H^2(X,\IZ)$, we would get $\phi(\delta)\in\NS(X)$. But then $\phi(\delta)$ is a MBM class on $X$ by \cite[Theorem 2.16]{AmerikVerbitskyMK}. Now by Remark \ref{utile} we have $(\pm\phi(\delta), k)>0$ for
all $k\in\cK_X$, in particular this is not zero. 

Conversely, let $\IC\omega\in \Omega_T^{\rho,\xi}\setminus\Delta$. By the surjectivity of the period map of $T$-polarized manifolds (see \cite[Proposition 3.2]{C4}), we know that there exists a $T$-polarized marked pair $(X,\phi)$ such that its period is $\IC\omega$; let $\IC\omega_X:=\phi(\IC\omega)$ be the line spanned by a symplectic holomorphic 2-form  on $X$. Define $\psi=\phi\circ\rho(\bar{\sigma})\circ\phi^{-1}$. It is an isometry of $H^2(X,\IZ)$ and it preserves the Hodge structure on $H^2(X,\IC)$ since 
$$
\psi(\omega_X)=\phi(\rho(\bar{\sigma})(\omega))=\phi(\xi\omega)=\xi\omega_X.
$$

It follows from Remark \ref{rho-monodromy} that $\psi\in \Mo^2(X)$.
We want now to use Markman's Torelli theorem to show that on $X$, or on a birational model of it, we can find a non-symplectic automorphism and thus we get the surjectivity of the restriction of the period map.

We study now the behaviour of 
$\psi$ with respect to the K\"ahler cone. 

{\bf First case}. If $\cK_X\cap \phi(T)\neq \emptyset $, this means that $\psi$ fixes a K\"ahler class. By Markman's Torelli theorem \cite[Theorem 1.3]{MarkmanTorelli}, $\psi$ is then induced by an automorphism $\sigma$, i.e. $\sigma^*=\psi$ and $(X,\phi)$ is $(\rho,T)$-polarized.

{\bf Second case}. If $\cK_X\cap \phi(T)= \emptyset$, we remark that for all $\delta\in \Delta(X)$, the invariant sublattice $H^2(X,\IZ)^{\psi}=\phi(T)$ is not contained in $\delta^{\perp}$. Otherwise there would exist a $\delta\in \Delta(X)$ orthogonal to $\phi(T)$, hence such that $\phi^{-1}(\delta)\in \Delta(S)$ (in particular is contained in $S$)  
but since $\delta\in\NS(X)$ we have also $(\phi^{-1}(\delta),\omega)=0$ which is in contradiction with $\omega\notin H_{\phi^{-1}(\delta)}\cap \Omega_T^{\rho,\xi}$. As a consequence, $\phi(T)$ is not contained in $\bigcup_{\delta\in\Delta(X)}\delta^{\perp}$.

In particular, this shows that the intersection of $H^2(X,\IZ)^{\psi}=\phi(T)$ with $\mathcal{C}_X\setminus\bigcup_{\delta\in\Delta(X)}\delta^{\perp}$ is non-empty. 
In other words, there exists $h\in H^2(X,\IZ)^{\psi}\cap \mathcal{K}$ for $\mathcal{K}$ a K\"ahler-Weyl chamber (see \cite[Definition 5.10]{MarkmanTorelli}) and so a birational map $f:X\dashrightarrow\tilde{X}$, $y$ a K\"ahler class on $\tilde{X}$ and $\tau$ a monodromy operator 
preserving the Hodge decomposition on $X$ such that $\tau(f^*(y))=h$. Now consider the isometry $\tilde{\psi}=g^{-1}\circ\psi\circ g$ on $H^2(\tilde{X},\IZ)$, for $g=\tau\circ f^*$: it is easy to see (with a computation as above) that this is a monodromy operator preserving the Hodge decomposition and it satisfies $\tilde{\psi}
(y)=y$, hence it fixes a K\"ahler class, and moreover $\tilde{X}$ has period $\omega$ and is $(\rho,T)$-polarized.

\end{proof} 

\begin{rem}
The previous theorem does not say that every birational model in a fiber of the previous period map admits a non symplectic automorphism. This is a priori not true, it becomes true if for a certain $X$ we have $\cC_X\subset \phi(T)_\mathbb{R}$. See the discussion
in the next section. 
\end{rem}
As a byproduct of the first part of the proof of Theorem \ref{surjectivity} above we obtain the following
\begin{cor}\label{iperpianiDeltaS}
Given a marked pair $(X,\phi)\in\mathcal{M}_L^+$, if $\mathcal{P}(X,\phi)\in\Delta$ then there is no automorphism of $X$ of prime order $p$ with invariant sublattice isometric to $T$.
\end{cor}
\begin{rem}
Observe that if $\dim S(\xi)=1$ then one has $\dim \Omega_T^{\rho,\xi}=0$, so that,
if this space is not empty, the period domain $\Omega_T^{\rho,\xi}$ consists exactly of one point and there exist finitely many $(\rho, T)$-polarized \ihskn. 
\end{rem}

Let $\mathcal{M}_T^{\rho,\xi}$ be the set of $(\rho,T)$-polarized pairs $(X,\phi)\in\mathcal{M}_L^+$ such that $\rho(\bar{\sigma})(\omega)=\xi \omega$, for $\omega$ the line in $L\otimes \IC$ defined by $\omega=\phi^{-1}(H^{2,0}(X))$. Theorem \ref{surjectivity} implies that the period map restricts to a surjection $\mathcal{P}:\mathcal{M}_T^{\rho,\xi}\ra\Omega_T^{\rho,\xi}\setminus\Delta$.

\section{Injectivity locus of the period map}\label{injectivityperiod}

In this section we use techniques developed in \cite{Joumaah} to construct an injective restriction of the period map. Given a non-symplectic automorphism $\sigma$ of order $p\geq 3$ on $(X,\phi)\in \mathcal{M}_T^{\rho,\xi}$, where $\phi$ is the marking, let $\mathcal{C}_X^{\sigma}=\{y\in\cC_X |, \sigma^*(y)=y\}$ denote the connected component of the fixed part of the positive cone that contains a K\"ahler class. Observe that this is surely non empty since there exists a $\sigma$-invariant ample  class on $X$. Consider the following chamber decomposition:
\begin{equation}\label{invdec}
 \mathcal{C}_X^{\sigma}\setminus \bigcup_{\delta\in \Delta_{\sigma }(X)}\ \delta^{\perp}
\end{equation}
where $\Delta_{\sigma }(X)=\lbrace \delta\in\Delta(X)| \sigma^*(\delta)=\delta\rbrace$. Clearly a priori this set contains  
less walls.  One can see $\mathcal{C}_X^{\sigma}$ also as the intersection:
$$
\mathcal{C}_X\cap\phi(T)_{\IR}.
$$
Remark also that since $\delta$ is $\sigma$-fixed then $\delta^{\perp}\subset H^{1,1}(X)\cap H^2(X,\IR)$ is $\sigma$-invariant so that $\delta^{\perp}\cap \phi(T)_{\IR}$ is non empty.
\begin{defi}
The {\it stable invariant K\"ahler cone} $\tilde{\mathcal{K}}_X^{\sigma}$ of $X$ is the chamber of (\ref{invdec}) containing the invariant K\"ahler cone $\mathcal{K}_X^{\sigma}:=\mathcal{K}_X\cap\mathcal{C}_X^{\sigma}$.
\end{defi}

One can show easily that for any $(X,\phi)\in\mathcal{M}_T^{\rho,\xi}$ we have $\phi(\Delta(T))=\Delta_{\sigma }(X)$. 
Let $C_T$ be the connected component of the cone $\lbrace x\in T_{\mathbb{R}}\, |\, q(x)>0\rbrace\subset L_{\IR}$ such that $\phi(C_T)=\mathcal{C}_{X}^{\sigma}$. Moreover, let $K(T)$ be a connected component of the chamber decomposition
\begin{equation}\label{decomp-C-T}
C_T\setminus\bigcup_{\delta\in \Delta(T)}\delta^{\perp}\subset T_{\IR}.
\end{equation}
Remark that here we work only with classes in $T_\IR$ and not with the whole lattice $L_{\IR}$: this will be  fundamental to
get an injective restriction of the period map. 

\begin{lemma}\label{forany} The following hold:
 \begin{enumerate}
  \item For any $(X_0,\phi_0)\in\mathcal{M}_T^{\rho,\xi}$ with non--symplectic automorphism $\sigma_0\in\Aut(X_0)$ such that $\mathcal{K}_{X_0}^{\sigma_0}\subset \phi(K(T))$,  then  $\phi(K(T))=\tilde{\mathcal{K}}_{X_0}^{\sigma_0}$.
  \item If $\NS(X)=\phi(T)$, then all MBM classes are $\sigma$-invariant. Hence we have that $\cC_X^{\sigma}=\cC_X$. Moreover if $\mathcal{K}_X^{\sigma}\subset \phi(K(T))$ we have 
$$
\mathcal{K}_{X}^{\sigma}= \phi(K(T))=\tilde{\mathcal{K}}_X^{\sigma}=\cK_X.
$$ 
\end{enumerate}
\end{lemma}
\begin{proof}
\begin{enumerate}
\item It follows immediately from $\phi(\Delta(T))=\Delta_{\sigma_0 }(X)$ and from the fact that $\phi_0(C_T)=\mathcal{C}_X^{\sigma_0}$.

\item If $\NS(X)=\phi(T)$, $\phi(\Delta(T))=\Delta_{\sigma }(X)=\Delta(X)$, hence the chamber decomposition is exactly the one defining the invariant K\"ahler cone $\mathcal{K}_X^{\sigma}$. Observe moreover that in this case $H^{1,1}(X)\cap H^2(X,\IR)$ is equal to $\phi(T)_{\IR}$ so it is invariant by $\sigma$. Since the positive cone is an open subset there we have that $\cC_X^{\sigma}=\cC_X$ in particular $\cK_X^{\sigma}=\cK_X$.
\end{enumerate}

\end{proof}

\begin{rem}
Observe that item $(2)$ in  Lemma \ref{forany}  corresponds to the ``generic'' case, since a
generic \ihskn in the moduli space has N\'eron-Severi group equal to $\phi(T)$.

\end{rem}

When the equality $\mathcal{K}_X^{\sigma}= \phi(K(T))$ fails, there exists an algebraic wall $\delta\in\Delta(X)$ which is not 
$\sigma$-fixed, i.e. it is not in $\Delta_{\sigma}(X)$, but $\delta^{\perp}\cap \cC_X^{\sigma} \not=\emptyset$. In particular, it does not appear in the decomposition \eqref{invdec}.
More precisely since we have that  $\mathcal{K}_X^{\sigma}\subsetneq \phi(K(T))$ it exists an algebraic wall $\delta\in\Delta(X)$ such that $\delta^{\perp}\cap \phi(K(T))\neq\emptyset$: such a wall $\delta$ is neither in $\phi(\Delta(T))$ nor in $\phi(\Delta(S))$.
The first fact is clear, if $\delta\in \phi(\Delta (S))$ then there would be an ample invariant class ($X$ is projective) that has zero intersection
with it, which is impossible (see the proof of Theorem \ref{surjectivity}). In fact $\phi^{-1}(\delta)$ belongs to the set $\Delta'(L)$ of elements $\nu\in\Delta(L)$ such that there is a decomposition $\nu=\nu_T+\nu_S$ with 
$\nu_T\in T_{\IQ}$ and $\nu_S\in S_{\IQ}$ and $q(\nu_T)<0$, $q(\nu_S)<0$ (see \cite{Joumaah}). Recall the following:
\begin{lemma} \cite[Lemma 7.6]{Joumaah}\label{characterization-Delta'}
 The following are equivalent:
 \begin{enumerate}
  \item $\delta\in\Delta'(L)$;
  \item $\sign(T\cap\delta^{\perp})=(1,\rank T-2)$ and $\sign(S\cap\delta^{\perp})=(2,\rank S-3)$ ;
  \item $\delta\notin T$, $\delta\notin S$, $\Omega_L\cap \IP(S)\cap\delta^{\perp}\neq\emptyset$ and $\lbrace x\in T_{\mathbb{R}}\, |\, q(x)>0\rbrace\cap \delta^{\perp}\neq\emptyset$.
 \end{enumerate}

\end{lemma}

The chambers of 
$$\phi(K(T))\setminus\bigcup_{\delta\in \Delta(X)\cap\phi(\Delta'(L))}\delta^{\perp}$$
 will turn out to correspond to elements $(Y,\eta)$ in the fiber of $\mathcal{P}^{-1}(\cP(X,\phi))$ satisfying $\cK_Y^{\sigma}\subset\eta(K(T))$.

We introduce the following definition:

\begin{defi}
 A $(\rho,T)$-polarized manifold $(X,\phi)$ is {\it $K(T)$-general} if $\phi(K(T))=\mathcal{K}_X^{\sigma}$, where $\sigma$ is the automorphism of order $p$ induced by $\rho(\bar{\sigma})$.
\end{defi}

Let $\mathcal{M}_{K(T)}^{\rho,\xi}$ be the set of $(\rho,T)$-polarized pairs $(X,\phi)\in \mathcal{M}_{T}^{\rho,\xi}$ that are $K(T)$-general. We denote by $\Delta'(K(T))$ the set of $\delta\in\Delta'(L)$ such that $\delta^{\perp}\cap K(T)\neq\emptyset$.

\begin{theorem}\label{injectivity}
 The period map induces a bijection $$\mathcal{P}:\mathcal{M}_{K(T)}^{\rho,\xi}\ra \Omega:=\Omega_T^{\rho,\xi}\setminus(\Delta\cup\Delta'),$$ where $\Delta$ is the one defined in Theorem \ref{surjectivity} and $\Delta':=\bigcup_{\delta\in \Delta'(K(T))} (H_{\delta}\cap \Omega_T^{\rho,\xi})$.
\end{theorem}
\begin{proof}
Given a $(X,\phi)\in\mathcal{M}_{K(T)}^{\rho,\xi}$, where $\phi$ denotes the marking, let $\omega:=\phi^{-1}(H^{2,0}(X))$ be its period. By Theorem \ref{surjectivity} (\ref{surjectivity-i}), we know that $ \omega\notin\Delta$. If we had $\omega\in H_{\delta}$ for a $\delta \in \Delta'(K(T))$, then by definition of $\NS(X)$ as orthogonal complement of $H^{2,0}(X)$ in $H^2(X,\IZ)$, we would get $\phi(\delta)\in\NS(X)$. But then $\phi(\delta)$ is a MBM class on $X$ by \cite[Theorem 2.16]{AmerikVerbitskyMK}. By construction, we would have $\phi(\delta)^{\perp}\cap \phi(K(T))\neq\emptyset$ and this immediately implies that $\mathcal{K}_X^{\sigma}\subsetneq \phi(K(T))$, against the assumption that $(X,\phi)$ is $K(T)$-general. Hence $ \omega\in\Omega$.

Given $\omega\in\Omega$, by Theorem \ref{surjectivity}, there exists $(X,\phi)\in\mathcal{M}_T^{\rho,\xi}$ such that its period is $\omega$; we want to show that there exists exactly one $K(T)$-general marked pair $(X',\phi')$ inside $\mathcal{P}^{-1}(\omega)$. Let $\sigma$ be the automorphism on $X$ that induces $\phi\circ\rho(\bar{\sigma})\circ\phi^{-1}$ on $H^2(X,\IZ)$. 
Either $\mathcal{K}_X^{\sigma}\subset \phi(K(T))$ or $\mathcal{K}_X^{\sigma}\cap \phi(K(T))=\emptyset$.

In the first case, suppose that we have $\mathcal{K}_X^{\sigma}\subsetneq \phi(K(T))$: this means that there exists a wall $\delta\in\Delta(X)\setminus\Delta_{\sigma}(X)$ such that $\delta^{\perp}\cap \phi(K(T))\neq\emptyset$. This implies that $\delta\in \phi(\Delta'(K(T)))\cap\Delta(X)$, against the assumption that $\omega \notin\Delta'$. Hence $(X,\phi)$ is $K(T)$-general.

Otherwise, if $\mathcal{K}_X^{\sigma}\cap \phi(K(T))=\emptyset$, let $\mathcal{K}'$ be a chamber of $\mathcal{C}_X\setminus \bigcup_{\delta\in \Delta(X)}\delta^{\perp}$ such that $\mathcal{K}'\cap\mathcal{C}_X^{\sigma}\subset\phi(K(T))$. As in the proof of Theorem \ref{surjectivity}, let $(X',\phi')\in \mathcal{M}_{T}^{\rho,\xi}$ be the birational model of $(X,\phi)$ such that $\tau(f^*(\mathcal{K}_{X'}))=\mathcal{K}'$ for $\tau$ a Hodge monodromy on $X$ and $f:X\dashrightarrow X'$ a birational map, moreover $\phi':=g^{-1}\circ\phi$ where $g=\tau\circ f^*$  and the non-symplectic automorphism on $X'$ is $\sigma'=g^{-1}\circ \sigma\circ g$. Then by construction  $\mathcal{K}_{X'}^{\sigma'}\subset \phi'(K(T))$ and we can repeat the above proof.

Suppose now that $\mathcal{P}(X_1,\phi_1)=\mathcal{P}(X_2,\phi_2)=\omega$ for two $K(T)$-general pairs $(X_1,\phi_1)$ and $ (X_2,\phi_2)$, and consider $f=\phi_2\circ\phi_1^{-1}$: it is a Hodge isometry and a parallel transport operator, by \cite[Theorem 3.2]{MarkmanTorelli}. Moreover, $f(\mathcal{K}_{X_1}^{\sigma_1})=f(\phi_1(K(T)))=\phi_2(K(T))=\mathcal{K}_{X_2}^{\sigma_2}$, where $\sigma_1$ and $\sigma_2$ denote the non-symplectic automorphisms on $X_1$ and $X_2$ respectively. We have shown that $f$ sends a K\"ahler class to a K\"ahler class so by Theorem \ref{HTT} we have that $f$ is induced by an isomorphism $\overline{f}:X_1\ra X_2$.

\end{proof}

\section{Complex ball quotients}\label{arithmetic-quot}
We study now how the period point $\cP(X,\phi)$ depends on the marking. By taking a monodromy operator $h\in \Mo^2(L)$ which is  the identity on $T$,  i.e. an element of $\lbrace g\in\Mo^2(L)|g_{|T}=\id_T\rbrace$, we get another marking $\phi\circ h^{-1}$ of the $T$-polarized \ihskn $(X,\iota)$. 

\begin{defi}
Let $(X_1, \eta_1)$ and $(X_2,\eta_2)$ be two $(\rho, T)$-polarized marked IHS manifolds and let $\sigma_i$ be a generator of $G\subset \Aut(X_i)$, for $i=1,2$.
 A {\it $(\rho, T)$-polarized parallel transport operator} is an isometry $f:H^2(X_1,\IZ)\ra H^2(X_2,\IZ)$ such that there exists a smooth and proper family $p:\mathcal{X}\ra D$, isomorphisms $\psi_i:X_i\ra\mathcal{X}_i$ and a path $\gamma:[0,1]\ra B$ with $\gamma(0)=b_1$, $\gamma(1)=b_2$, such that parallel transport in $R^2\pi_*\IZ$ along $\gamma$ induces $\tilde{f}:=(\psi_2)_*\circ f\circ(\psi_1)^*$, and a holomorphic map $F:\mathcal{X}\ra\mathcal{X}$ such that $F_b$ is an automorphism for all $b\in B$, $F_{b_i}$ induces $\sigma_i^* $ and $F_{b_1}^*=\tilde{f}^{-1}\circ F_{b_2}^*\circ \tilde{f}$.
 
 In particular, a {\it $(\rho, T)$-polarized monodromy operator} is a monodromy operator on $X$ satisfying $F_{b_1}^*\circ \tilde{f}=\tilde{f}\circ F_{b_1}^*$.
 \end{defi}

Let $\Gamma_T$ be  $\alpha(\lbrace g\in\Mo^2(L)|g_{|T}=\id_T\rbrace)$, with $\alpha: O(L)\ra O(S)$ the restriction map. Let this group act on  $\Omega_T^{\rho,\xi}$; the stabilizer of $\Omega_T^{\rho,\xi}$ is equal to the image $\Gamma_{T}^{\rho,\xi}\subset O(S)$, via the restriction map $\alpha$, of the group of $(\rho,T)$-polarized monodromy operators
$$
\Mo^2(T,\rho)=\lbrace g\in\Mo^2(L)|g_{|T}=\id_T,\ g\circ \rho(\bar\sigma)=\rho(\bar\sigma)\circ g\rbrace\subset O(L).
$$

\begin{prop}\label{equivariance}
 The group $\Mo^2(T,\rho)$ acts on the set $\mathcal{M}_{K(T)}^{\rho,\xi}$. Moreover, the restriction of the period $\mathcal{P}:\mathcal{M}_{K(T)}^{\rho,\xi}\ra \Omega$ is equivariant with respect to the action of $\Mo^2(T,\rho)$ and $\Gamma_{T}^{\rho,\xi}$, so it induces a bijection between the quotients.
\end{prop}
\begin{proof}
 Given $(X,\phi)\in\mathcal{M}_{K(T)}^{\rho,\xi}$, we want to show that $(X,\phi\circ g^{-1})\in\mathcal{M}_{K(T)}^{\rho,\xi}$. Indeed, $(\phi\circ g^{-1})_{|T}=\iota$, $\phi\circ g^{-1}\circ\rho(\bar{\sigma})\circ g\circ\phi^{-1}=\phi\circ g^{-1}\circ  g\circ\rho(\bar{\sigma})\circ \phi^{-1}=\bar{\sigma}^*$ and $\phi(g^{-1}(K(T)))=\phi(K(T))$ because $g^{-1}_{|T}=\id_T$.
 
 The equivariance is obvious.
\end{proof}

\begin{lemma}\label{arithmeticGroup}
 The group $\Gamma_{T}^{\rho,\xi}$ is an arithmetic group.
\end{lemma}
\begin{proof}
We remark first that $\Gamma_{T}^{\rho,\xi}=\lbrace g\in\Gamma_T| g\circ \rho(\bar\sigma)=\rho(\bar\sigma)\circ g\rbrace\subset O(S)$.
The group $Z=\lbrace g\in O(S,\IQ)| g\circ\rho(\bar\sigma)=\rho(\bar\sigma)\circ g\rbrace$ is a linear algebraic group over $\IQ$, because the condition $g\circ\rho(\bar\sigma)=\rho(\bar\sigma)\circ g$ gives polynomial equations in the coefficients of the matrix associated to $g$. Hence $Z(\IZ)=Z\cap O(S)$ is an arithmetic group. We know from \cite[Proposition 3.5]{C4} and its generalization in \cite{C5} that $\Gamma_T$ is of finite index inside $O(S)$ and this implies that $\Gamma_{T}^{\rho,\xi}=Z(\IZ)\cap \Gamma_T$ is of finite index inside $O(S)\cap Z(\IZ)=Z(\IZ)$, thus it is an arithmetic group.
\end{proof}

A straightforward generalization of \cite[Lemma 7.7]{Joumaah} gives the following result.

\begin{lemma}\label{loc-finiteness}
 The collections of hyperplanes $\Delta$ and $\Delta'$ in $\Omega_L\cap\IP(S)$ are locally finite.
\end{lemma}
\begin{proof}
 The local finiteness of $\Delta$ is proven in  \cite[Lemma 7.7]{Joumaah}, where it is also proven for $\Delta'$ in the case in which $T$ is the invariant sublattice of a non-symplectic involution. That proof can be easily generalized in the following way: given $\delta\in \Delta'(L)$, remark that there exists an integer $d$ such that $d\delta=\delta_T+\delta_S$ with $\delta_T\in T$ and $\delta_S\in S$. Indeed, $T\oplus S$ is a sublattice of finite index in $L$, and it is enough to take $d:=|L/(T\oplus S)|$ to obtain such a decomposition (in particular, in our case $d=p^a$ by \cite[Lemma 4.3]{BNWS}). Then, by Lemma \ref{characterization-Delta'} we deduce $q(\delta_T)<0$ and $q(\delta_S)<0$, but we know that there are only finitely many possible values for $q(\delta)<0$ for an MBM class, so this holds as well for $\delta_S$. 
 
 It is now enough to remark that $H_{\delta}=H_{\delta_S}$ in  $\Omega_L\cap\IP(S)$ and that $\Gamma_T$ acts on the set of hyperplanes $H_{\delta_S}$ with a proper and discontinuous action, so that every orbit is closed, and hence locally finite.
\end{proof}

\begin{cor}\label{algebraicModuli}
 The complex ball quotient $\mathcal{M}_{K(T)}^{\rho,\xi}/\Mo^2(T,\rho)\cong \Omega/\Gamma_{T}^{\rho,\xi}$ parametrizes isomorphism classes of $K(T)$-general $(\rho,T)$-polarized \ihskn, and it is a quasi-projective variety of dimension $\dim S(\xi)-1$.
\end{cor}
\begin{proof}
The first part of the statement is a direct consequence of Theorem \ref{injectivity} and Proposition \ref{equivariance}.

The sets of hyperplanes $\Delta$ and $\Delta':=\bigcup_{\delta\in \Delta'(K(T))} H_{\delta}$ are locally finite in the period domain $\Omega_{T^\perp}$ of $T$-polarized \ihskn by Lemma \ref{loc-finiteness}. Now $\Omega_T^{\rho, \xi}$ is contained in $\Omega_{T^\perp}$ so that $\Delta$ and $\Delta'$ remain locally finite. Then by Proposition \ref{arithmeticGroup} and Baily-Borel's Theorem \cite{BailyBorel} the arithmetic quotient $\Omega/\Gamma_{T}^{\rho,\xi}$ is a quasi-projective variety of dimension $\dim S(\xi)-1$. 
\end{proof}

We have seen in Corollary \ref{iperpianiDeltaS} that, if $\mathcal{P}(X,\phi)\in\Delta$, there is no automorphism on $X$ of prime order $p$ with invariant sublattice $T$. The next statement explains what happens when the period belongs to $\Delta'$.

We need to recall the following notation: given a period point $\pi\in\Omega_L$, the group of monodromies which are isomorphisms of Hodge structures is the same for all marked pairs in $\mathcal{P}^{-1}(\pi)$; we denote it $\Mo^2_{\mathrm{Hdg}}(\pi)$.

\begin{prop}\label{Delta'}
If $\pi\in\Delta'\setminus\Delta$, each element $(X,\phi)$ of $\mathcal{P}^{-1}(\pi)$ such that $\phi(K(T))\cap \mathcal{K}_X\neq\emptyset$ has an automorphism of order $p$ with invariant sublattice $T$ but it is not $K(T)$-general. There is a bijection between $\Mo^2_{\mathrm{Hdg}}(\pi)$-orbits of elements $(Y,\eta)$ of $\mathcal{P}^{-1}(\pi)$ such that $\eta(K(T))\cap \mathcal{K}_Y\neq\emptyset$ and $\Mo^2_{\mathrm{Hdg}}(X)$-orbits of chambers of $U:=\phi(K(T))\setminus\bigcup_{\delta\in \Delta(X)\cap\phi(\Delta'(K(T)))}\delta^{\perp}$,
for any $(X,\phi)\in\mathcal{P}^{-1}(\pi)$.  
\end{prop}
\begin{proof}
Let $(Y,\eta)\in\mathcal{P}^{-1}(\pi)$ be such that $\eta(K(T))\cap \mathcal{K}_Y\neq\emptyset$. It follows from the proof of Theorem \ref{surjectivity} that there is an automorphism $\sigma_Y$ of order $p$ with invariant sublattice $T$, since by assumption there is a K\"ahler class invariant under the action of $\eta\circ\rho(\bar{\sigma})\circ\eta^{-1}$. On the other hand, $\pi\in\Delta'\setminus\Delta$ implies that there is $\delta\in\Delta'(K(T))$ such that $\eta(\delta)\in \NS(Y)$. In particular, $\eta(\delta)^{\perp}\cap \eta(K(T))\neq\emptyset$ and there is more than one chamber in $U$; this immediately implies that $(Y,\eta)$ is not $K(T)$-general.

Given $(Y,\eta)\in\mathcal{P}^{-1}(\pi)/\Mo^2_{\mathrm{Hdg}}(\pi)$ such that $\eta(K(T))\cap \mathcal{K}_Y\neq\emptyset$, we define $\beta(Y,\eta)=[\mathcal{K}_Y^{\sigma_Y}]$, the equivalence class of $\mathcal{K}_Y^{\sigma_Y}$ with respect to the action of $\Mo^2_{\mathrm{Hdg}}(Y)$, and we will show that $\beta$ is the desired bijection. It is clearly well-defined.

Suppose that $\beta(Y,\eta)=\beta(Y',\eta')$, or equivalently $\eta'(\eta^{-1}(\mathcal{K}_Y^{\sigma_Y}))=\mathcal{K}_{Y'}^{\sigma_{Y'}}$. Since $\mathcal{P}(Y,\eta)=\mathcal{P}(Y',\eta')$, it follows from the global Torelli theorem \cite[Theorem 1.3]{MarkmanTorelli} that $\eta'\circ\eta^{-1}$ is induced by an isomorphism $f:Y'\ra Y$.

Let $(X,\phi)\in\mathcal{P}^{-1}(\pi)$ such that $\phi(K(T))\cap \mathcal{K}_X\neq\emptyset$ (such $(X,\phi)$ can be constructed as in the proof of Theorem \ref{injectivity}) and let $\mathcal{K}_0$ be a chamber of $U$. 

Remark that $$U=\phi(K(T))\cap\left(\mathcal{C}_X^{\sigma_X}\setminus\bigcup_{\delta\in \Delta(X)}\delta^{\perp}\right)=\phi(K(T))\cap \left(\mathcal{C}_X\setminus\bigcup_{\delta\in \Delta(X)}\delta^{\perp}\right)$$ Indeed, the elements $\delta\in\Delta(X)$ such that $\delta^{\perp}\cap\phi(K(T))\neq\emptyset$ are by definition those in $\phi(\Delta'(K(T)))$ and $\mathcal{C}_X^{\sigma_X}\cap\phi(K(T)) = \mathcal{C}_X\cap\phi(K(T))$. This tells us that there is one chamber $C_0$ of  $\mathcal{C}_X\setminus\bigcup_{\delta\in \Delta(X)}\delta^{\perp}$, and in fact a unique one, such that $C_0\cap\phi(K(T))=\mathcal{K}_0$. By \cite[Proposition 5.14 ]{MarkmanTorelli} we know that there is $(Y,\eta)\in\mathcal{P}^{-1}(\pi)$ such that $\mathcal{K}_Y=\eta(\phi^{-1}(C_0))$ and, in particular, $\eta(K(T))\cap \mathcal{K}_Y\neq\emptyset$. Moreover, $\beta(Y,\eta)=\mathcal{K}_Y^{\sigma_Y}=\mathcal{K}_0$.
\end{proof}
\section{Examples and final remarks}
\subsection{A special example: $T=\langle 6 \rangle$} If $(X,\phi)$ is a point in $\cM^{\rho,\xi}_T$, then $\cC_X^{\sigma}=\cC_X\cap\phi(T)_{\IR}=\phi(T)_{\IR}$, where $\sigma$ denotes the non-symplectic automorphism of order three on $X$. In particular, this is a one-dimensional space
and we have also $\cK_X^{\sigma}=\phi(T)_\IR$, so that $\phi(K(T))=\phi(T)_{\IR}=\cK_X^{\sigma}$. This means that any point in $\cM^{\rho,\xi}_T$ is automatically $K(T)$-general, so that $\cM^{\rho,\xi}_T=\cM^{\rho,\xi}_{K(T)}$. On the other hand, here clearly $\Delta'=\emptyset$ so the period map:
$$
\mathcal{M}_T^{\rho,\xi}/\Mo^2(\rho,T)\ra (\Omega_T^{\rho,\xi}\setminus\Delta)/\Gamma_T^{\rho,\xi}
$$
is in fact bijective. 

This ball quotient has been already thoroughly studied by Allcock, Carlson and Toledo in \cite{Allcock-Carlson-Toledo}, where they study the moduli space of smooth cubic threefolds by associating to each such threefold $Y\subset \IP^4$ a triple cover of $\IP^4$ ramified exactly along $Y$. This construction gives exactly the family of smooth cubic fourfolds
\[
V_1:L_3(x_0,\dots,x_4)+x_5^3=0,\]
where $L_3$ defines a smooth cubic $Y$; we have shown in \cite[Example 6.4]{BCS} that the Fano varieties of lines $F(V_1)$ of cubics $V_1$ give a $10$-dimensional family of \ihsk admitting a non-symplectic automorphism of order three with invariant sublattice $T=\langle 6 \rangle$.

The hyperplane arrangement $\Delta$ coincides exactly with the union of the two arrangements $\mathcal{H}_{\Delta}$ and $\mathcal{H}_c$ corresponding respectively to discriminant and chordal cubics (compare with the definition given in \cite{Allcock-Carlson-Toledo} before Theorem 3.7).

We briefly explain here why: the hyperplanes in $\mathcal{H}_{\Delta} \cup \mathcal{H}_c$ are the orthogonal to the roots, of square $3$, of the Eisenstein lattice associated to $S$, and the quadratic form $q_S$ on $S$ is then recovered by the quadratic form $q_{\mathcal{E}}$ via $q_S(\delta)=-\frac{2}{3}q_{\mathcal{E}}(\delta)$. Thus, in $\IP(S\otimes\IC)$ they correspond exactly to hyperplanes orthogonal to the roots, of square $-2$, in $S$. On the other hand, we know, by results of Bayer, Hassett and Tschinkel \cite{BayerHassettTschinkel} and of Mongardi \cite{MongardiCone}, that $\delta\in\Delta(S)$ are elements in $S$ of square $-2$ or of square $-10$ and divisibility $\mathrm{div}\delta=2$, and these second ones cannot exist in a lattice $S$ of determinant $3$.

\subsection{Two different chambers}The theory of \S \ref{injectivityperiod} also offers an interesting explanation of \cite[Remark 7.7]{BCS}. In that case, $p=3$ and $T=U\oplus A_2^{\oplus 5}\oplus\langle -2\rangle$ and we were able to construct to families of examples, one of natural automorphisms on the Hilbert scheme of two points of a $K3$ surface and one on a family of Fano varieties of smooth cubic foufolds. The peculiarity of this example is that, for the first time, it gives an example of two families with automorphisms having the same action on cohomology but non-isomorphic fixed loci.

In this example, we obtain a four-dimensional complex ball quotient $\Omega$.
Theorem \ref{ball}, and his proof, shows that, for every $\pi\in\Omega$, there exists a marked pair $(\Sigma^{[2]},\phi)\in\mathcal{P}^{-1}(\pi)$, with $\Sigma$ a smooth $K3$ surface, with the natural automorphism of order three acting on it. 

On the other hand, also the family of Fano varieties of cubics of the form 
\[
V_2:L_3(x_0,x_1,x_2,x_3)+M_3(x_4,x_5)=0\]
gives an open dense subset of $\mathcal{M}_T^{\rho,\xi}$ of maximal dimension four.

Our interpretation of this phenomenon is that over the very general period point $\pi\in\Omega$, the fibre $\mathcal{P}^{-1}(\pi)$ will contain a natural marked pair and at least one marked pair coming from the Fano variety $F(V_2)$. Corollary \ref{algebraicModuli} then implies that there exists two chambers $K_1$ and $K_2$ of (\ref{decomp-C-T}) such that the general element in $\mathcal{M}_{K_1}^{\rho,\xi}$ is a Hilbert scheme of two points of a smooth $K3$ surface with a natural automorphism, while the general element in $\mathcal{M}_{K_2}^{\rho,\xi}$ is the Fano variety of a cubic in $V_2$.

\subsection{Rationality} If we restrict once more to the assumptions of Theorem \ref{ball}, namely if we suppose that $T=\bar{T}\oplus\langle -2\rangle$, with $\bar{T}\subset L_{K3}$, and we consider isomorphism classes of $K(T)$-general $(\rho,T)$-polarized \ihskcom Corollary \ref{algebraicModuli} gives an isomorphism $\mathcal{M}_{K(T)}^{\rho,\xi}/\Mo^2(T,\rho)\cong \Omega/\Gamma_{T}^{\rho,\xi}$.

On the other hand, the arithmetic group $\Gamma_{T}^{\rho,\xi}$ has index two inside the arithmetic group $\Gamma_{K3,\bar{T}}^{\rho,\xi}:=\lbrace g\in O(S)| g\circ \rho(\bar\sigma)=\rho(\bar\sigma)\circ g\rbrace\subset O(S)$. Hence,
$\Omega/\Gamma_{T}^{\rho,\xi}$ is generically a double cover of $\Omega/\Gamma_{K3,\bar{T}}^{\rho,\xi}$, and this is a Zariski-open in $\Omega_{K3}/\Gamma_{K3,\bar{T}}^{\rho,\xi}$, where $\Omega_{K3}:=\Omega_{T}^{\rho,\xi}\setminus\Delta$; it is classically known (see \cite{DK}) that $\Omega_{K3}/\Gamma_{K3,\bar{T}}^{\rho,\xi}$ parametrizes isomorphism classes of $(\rho,\bar{T})$-polarized $K3$ surfaces.

It follows that, even in the case of \ihskcom it is not possible to deduce any information about the rationality of $\mathcal{M}_{K(T)}^{\rho,\xi}/\Mo^2(T,\rho)$ from the work of Ma, Ohashi and Taki \cite{Ma-Ohashi-Taki} about the rationality of $\Omega_{K3}/\Gamma_{K3,\bar{T}}^{\rho,\xi}$ for $p=3$.

The rationality problem is still open for all prime orders $p>3$ also in the case of $K3$ surfaces.


\bibliographystyle{amsplain}
\bibliography{BibBallQuot}

\end{document}